\documentclass[12pt,twoside]{article}
\usepackage{amsmath,amssymb}
\usepackage{amsthm}
\usepackage{upref}
\usepackage{citeref}
\usepackage{graphicx,amssymb}
\usepackage{color}
\numberwithin{equation}{section}
\newtheorem*{theorem*}{Main Theorem}
\theoremstyle{plain}
\newtheorem{theorem}{Theorem}[section]

\newtheorem{definition}[theorem]{Definition}
\newtheorem{corollary}[theorem]{Corollary}

\theoremstyle{definition}
\newtheorem{remark}[theorem]{Remark}

\begin{document}

\newcommand{\eq}{equation}
\newcommand{\real}{\ensuremath{\mathbb R}}
\newcommand{\comp}{\ensuremath{\mathbb C}}
\newcommand{\rn}{\ensuremath{{\mathbb R}^n}}
\newcommand{\tn}{\ensuremath{{\mathbb T}^n}}
\newcommand{\rnp}{\ensuremath{\real^n_+}}
\newcommand{\rnn}{\ensuremath{\real^n_-}}
\newcommand{\Rn}{\ensuremath{{\mathbb R}^{n-1}}}
\newcommand{\Zn}{\ensuremath{{\mathbb Z}^{n-1}}}
\newcommand{\no}{\ensuremath{\nat_0}}
\newcommand{\ganz}{\ensuremath{\mathbb Z}}
\newcommand{\zn}{\ensuremath{{\mathbb Z}^n}}
\newcommand{\zom}{\ensuremath{{\mathbb Z}_{\Om}}}
\newcommand{\zOm}{\ensuremath{{\mathbb Z}^{\Om}}}
\newcommand{\As}{\ensuremath{A^s_{p,q}}}
\newcommand{\Bs}{\ensuremath{B^s_{p,q}}}
\newcommand{\Fs}{\ensuremath{F^s_{p,q}}}
\newcommand{\Fsr}{\ensuremath{F^{s,\rloc}_{p,q}}}
\newcommand{\nat}{\ensuremath{\mathbb N}}
\newcommand{\Om}{\ensuremath{\Omega}}
\newcommand{\di}{\ensuremath{{\mathrm d}}}
\newcommand{\sn}{\ensuremath{{\mathbb S}^{n-1}}}
\newcommand{\Ac}{\ensuremath{\mathcal A}}
\newcommand{\Acs}{\ensuremath{\Ac^s_{p,q}}}
\newcommand{\Bc}{\ensuremath{\mathcal B}}
\newcommand{\Cc}{\ensuremath{\mathcal C}}
\newcommand{\cc}{{\scriptsize $\Cc$}${}^s (\rn)$}
\newcommand{\ccd}{{\scriptsize $\Cc$}${}^s (\rn, \delta)$}
\newcommand{\Fc}{\ensuremath{\mathcal F}}
\newcommand{\Lc}{\ensuremath{\mathcal L}}
\newcommand{\Mc}{\ensuremath{\mathcal M}}
\newcommand{\Ec}{\ensuremath{\mathcal E}}
\newcommand{\Pc}{\ensuremath{\mathcal P}}
\newcommand{\Efr}{\ensuremath{\mathfrak E}}
\newcommand{\Mfr}{\ensuremath{\mathfrak M}}
\newcommand{\Mbf}{\ensuremath{\mathbf M}}
\newcommand{\Dbb}{\ensuremath{\mathbb D}}
\newcommand{\Lbb}{\ensuremath{\mathbb L}}
\newcommand{\Pbb}{\ensuremath{\mathbb P}}
\newcommand{\Qbb}{\ensuremath{\mathbb Q}}
\newcommand{\Rbb}{\ensuremath{\mathbb R}}
\newcommand{\vp}{\ensuremath{\varphi}}
\newcommand{\hra}{\ensuremath{\hookrightarrow}}
\newcommand{\supp}{\ensuremath{\mathrm{supp \,}}}
\newcommand{\ssupp}{\ensuremath{\mathrm{sing \ supp\,}}}
\newcommand{\dist}{\ensuremath{\mathrm{dist \,}}}
\newcommand{\unif}{\ensuremath{\mathrm{unif}}}
\newcommand{\ve}{\ensuremath{\varepsilon}}
\newcommand{\vk}{\ensuremath{\varkappa}}
\newcommand{\vr}{\ensuremath{\varrho}}
\newcommand{\pa}{\ensuremath{\partial}}
\newcommand{\oa}{\ensuremath{\overline{a}}}
\newcommand{\ob}{\ensuremath{\overline{b}}}
\newcommand{\of}{\ensuremath{\overline{f}}}
\newcommand{\LA}{\ensuremath{L^r\!\As}}
\newcommand{\LcA}{\ensuremath{\Lc^{r}\!A^s_{p,q}}}
\newcommand{\LcdA}{\ensuremath{\Lc^{r}\!A^{s+d}_{p,q}}}
\newcommand{\LcB}{\ensuremath{\Lc^{r}\!B^s_{p,q}}}
\newcommand{\LcF}{\ensuremath{\Lc^{r}\!F^s_{p,q}}}
\newcommand{\Lf}{\ensuremath{L^r\!f^s_{p,q}}}
\newcommand{\La}{\ensuremath{\Lambda}}
\newcommand{\Lob}{\ensuremath{L^r \ob{}^s_{p,q}}}
\newcommand{\Lof}{\ensuremath{L^r \of{}^s_{p,q}}}
\newcommand{\Loa}{\ensuremath{L^r\, \oa{}^s_{p,q}}}
\newcommand{\Lcoa}{\ensuremath{\Lc^{r}\oa{}^s_{p,q}}}
\newcommand{\Lcob}{\ensuremath{\Lc^{r}\ob{}^s_{p,q}}}
\newcommand{\Lcof}{\ensuremath{\Lc^{r}\of{}^s_{p,q}}}
\newcommand{\Lca}{\ensuremath{\Lc^{r}\!a^s_{p,q}}}
\newcommand{\Lcb}{\ensuremath{\Lc^{r}\!b^s_{p,q}}}
\newcommand{\Lcf}{\ensuremath{\Lc^{r}\!f^s_{p,q}}}
\newcommand{\id}{\ensuremath{\mathrm{id}}}
\newcommand{\tr}{\ensuremath{\mathrm{tr\,}}}
\newcommand{\trd}{\ensuremath{\mathrm{tr}_d}}
\newcommand{\trL}{\ensuremath{\mathrm{tr}_L}}
\newcommand{\ext}{\ensuremath{\mathrm{ext}}}
\newcommand{\re}{\ensuremath{\mathrm{re\,}}}
\newcommand{\rank}{\ensuremath{\mathrm{rank\,}}}
\newcommand{\Rea}{\ensuremath{\mathrm{Re\,}}}
\newcommand{\Ima}{\ensuremath{\mathrm{Im\,}}}
\newcommand{\loc}{\ensuremath{\mathrm{loc}}}
\newcommand{\rloc}{\ensuremath{\mathrm{rloc}}}
\newcommand{\osc}{\ensuremath{\mathrm{osc}}}
\newcommand{\pr}{\pageref}
\newcommand{\wh}{\ensuremath{\widehat}}
\newcommand{\wt}{\ensuremath{\widetilde}}
\newcommand{\ol}{\ensuremath{\overline}}
\newcommand{\os}{\ensuremath{\overset}}
\newcommand{\Li}{\ensuremath{\overset{\circ}{L}}}
\newcommand{\Ai}{\ensuremath{\os{\, \circ}{A}}}
\newcommand{\Ci}{\ensuremath{\os{\circ}{\Cc}}}
\newcommand{\dom}{\ensuremath{\mathrm{dom \,}}}
\newcommand{\SA}{\ensuremath{S^r_{p,q} A}}
\newcommand{\SB}{\ensuremath{S^r_{p,q} B}}
\newcommand{\SF}{\ensuremath{S^r_{p,q} F}}
\newcommand{\Hc}{\ensuremath{\mathcal H}}
\newcommand{\Nc}{\ensuremath{\mathcal N}}
\newcommand{\Lci}{\ensuremath{\overset{\circ}{\Lc}}}
\newcommand{\bmo}{\ensuremath{\mathrm{bmo}}}
\newcommand{\BMO}{\ensuremath{\mathrm{BMO}}}
\newcommand{\cm}{\\[0.1cm]}
\newcommand{\Aa}{\ensuremath{\os{\, \ast}{A}}}
\newcommand{\Ba}{\ensuremath{\os{\, \ast}{B}}}
\newcommand{\Fa}{\ensuremath{\os{\, \ast}{F}}}
\newcommand{\Aas}{\ensuremath{\Aa{}^s_{p,q}}}
\newcommand{\Bas}{\ensuremath{\Ba{}^s_{p,q}}}
\newcommand{\Fas}{\ensuremath{\Fa{}^s_{p,q}}}
\newcommand{\Ca}{\ensuremath{\os{\, \ast}{{\mathcal C}}}}
\newcommand{\Cas}{\ensuremath{\Ca{}^s}}
\newcommand{\Car}{\ensuremath{\Ca{}^r}}
\newcommand{\bl}{$\blacksquare$}

\begin{center}
{\Large Spectral theory for\\fractal pseudodifferential operators}
\\[1cm]
{Hans Triebel}
\\[0.2cm]
Institut f\"{u}r Mathematik\\
Friedrich--Schiller--Universit\"{a}t Jena\\
07737 Jena, Germany
\\[0.1cm]
email: hans.triebel@uni-jena.de
\end{center}

\begin{center}
{\em In memory of Albrecht Pietsch} ({\em 1934--2024})
\end{center} 

\begin{abstract}
\noindent The paper deals with the distribution of eigenvalues of the compact fractal pseudodifferential operator $T^\mu_\tau$,
\[
\big( T^\mu_\tau f\big)(x) = \int_{\rn} e^{-ix\xi} \, \tau(x,\xi) \, \big( f\mu  \big)^\vee (\xi) \, \di \xi, \qquad x\in \rn,
\]
in suitable special Besov spaces $B^s_p (\rn) = B^s_{p,p} (\rn)$, $s>0$, $1<p<\infty$. Here $\tau(x,\xi)$ are the symbols of (smooth)
pseudodifferential operators belonging to appropriate H\"{o}rmander classes $\Psi^\sigma_{1, \vr} (\rn)$, $\sigma <0$,   $0 \le \vr 
\le 1$ (including the exotic case $\vr =1$) whereas $\mu$ is the Hausdorff measure  of a compact $d$--set $\Gamma$ in \rn, $0<d<n$. 
This extends  previous assertions for the positive-definite selfadjoint fractal differential operator $(\id - \Delta)^{\sigma/2} \mu$
based on Hilbert space arguments in the context of suitable Sobolev spaces $H^s (\rn) = B^s_2 (\rn)$. We collect the outcome in the 
{\bfseries Main Theorem} below. Proofs are based on estimates for the entropy numbers  of the compact trace operator
\[
\tr_\mu: \quad B^s_p (\rn) \hra L_p (\Gamma, \mu), \quad s>0, \quad 1<p<\infty.
\]
We add at the end of the paper a few personal reminiscences illuminating the role of Pietsch in connection with the creation of 
approximation numbers and entropy numbers.
\end{abstract}

\indent eigenvalues, entropy numbers 


\section[Spaces, pseudodiff. op.]{Spaces and smooth pseudodifferential operators}  \label{S1}
\subsection[Notation, spaces]{Basic notation and spaces}   \label{S1.1}
We fix some notation. Let $\nat$ be the collection of all natural numbers and $\no = \nat \cup \{0 \}$. Let $\rn$ be Euclidean 
$n$--space, where $n \in \nat$. Put $\real = \real^1$ whereas $\comp$ is the complex plane. As usual $\ganz$ is the collection of all
integers, and $\zn$, where $n\in \nat$, denotes the lattice of all points $m = (m_1, \ldots, m_n) \in \rn$ with $m_j \in \ganz$, $j=
1, \ldots,n$. Let $\nat^n_0$, where $n\in \nat$, be the set of all multi--indices, $\alpha = (\alpha_1, \ldots, \alpha_n)$ with
$\alpha_j \in \no$ and $|\alpha| = \sum^n_{j=1} \alpha_j$. As usual, derivatives are abbreviated by
\begin{\eq}    \label{1.1}
D^\alpha = \frac{\pa^{|\alpha|}}{\pa x^{\alpha_1}_1 \ldots \pa x^{\alpha_n}_n}, \qquad \alpha \in \nat^n_0, \quad x \in \rn,
\end{\eq}
and
\begin{\eq}   \label{1.2}
\xi^\alpha = \xi^{\alpha_1}_1 \cdots \xi^{\alpha_n}_n, \qquad \alpha \in \nat^n_0, \quad \xi \in \rn.
\end{\eq}
For $x\in \rn$, $y\in \rn$, let
\begin{\eq}   \label{1.3}
xy = \sum^n_{j=1} x_j y_j, \qquad x= (x_1, \ldots, x_n), \quad y = (y_1, \ldots, y_n),
\end{\eq}
and $|x| = \sqrt{xx}$. If $a\in \real$ then $a_+ = \max (a,0)$. 

Let $M$ be a Lebesgue measurable set in $\rn$ where $|M|$ is the related Lebesgue measure, $0 \le |M| \le \infty$. Then $L_p (M)$, $0<p <\infty$,
is the usual quasi--Banach space of all complex--valued Lebesgue--measurable functions in $M$ such that
\begin{\eq}   \label{1.4}
\| f \, | L_p (M)\| = \Big( \int_M |f(x)|^p \, \di x \Big)^{1/p} <\infty,
\end{\eq}
complemented by
\begin{\eq}   \label{1.5}
\|f \, |L_\infty(M) \| = \inf \big\{ N: \, |\{x\in M: \, |f(x)| >N \} | =0 \big\}.
\end{\eq}

Let $\big(S(\rn), S'(\rn) \big)$ be the usual dual pairing, where $S(\rn)$ is the Schwartz space of all complex--valued rapidly 
decreasing $C^\infty$ functions
in $\rn$ and $S'(\rn)$ is its topological dual, the space of tempered distributions on \rn. Let $\vp \in S(\rn).$ Then
\begin{\eq}   \label{1.6}
\wh{\vp} (\xi) = \big(F \vp \big)(\xi) = (2\pi)^{-n/2} \int_{\rn}  e^{-ix \xi} \, \vp (x) \, \di x, \qquad \xi \in \rn,
\end{\eq}
is the Fourier transform of \vp, and
\begin{\eq}   \label{1.7}
\vp^\vee (\xi) = \big( F^{-1} \vp \big)(\xi) =(2 \pi)^{-n/2} \int_{\rn} e^{i x \xi} \vp (x) \, \di x,\qquad \xi \in \rn,
\end{\eq}
the related inverse Fourier transform. Both $F$ and $F^{-1}$ are extended to $S'(\rn)$ in the standard way. We assume that the reader
is familiar with basic properties of the Fourier transform, subject of many textbooks, including \cite{HT08} and the forthcoming book
\cite{HST}.

Next we recall briefly the Fourier--analytical definition of some function spaces in \rn. Let $\vp_0 \in S(\rn)$, 
\begin{\eq}   \label{1.8}
\vp_0 (x) =1 \ \text{if $|x| \le 1$} \qquad \text{and} \qquad \vp_0 (x) =0 \ \text{if $|x| \ge 3/2$},
\end{\eq}
and let
\begin{\eq}   \label{1.9}
\vp_k (x) = \vp_0 \big( 2^{-k} x) - \vp_0 \big( 2^{-k+1} x \big), \qquad x \in \rn, \quad k \in \nat.
\end{\eq}
Since
\begin{\eq}    \label{1.10}
\sum^\infty_{j=0} \vp_j (x) = 1 \qquad \text{for} \quad x\in \rn,
\end{\eq}
the $\vp_j$ form a dyadic resolution of unity. The entire analytic functions $(\vp_j \wh{f} )^\vee (x)$ make sense pointwise in $\rn$ for any $f \in S' (\rn)$. 

\begin{definition}   \label{D1.1}
Let $\vp = \{ \vp_j \}^\infty_{j=0}$ be the above dyadic resolution  of unity. Let
\begin{\eq}   \label{1.11}
0<p,q \le \infty  \qquad  \text{and} \quad s \in \real.
\end{\eq}
Then $\Bs (\rn)$ is the collection of all $f \in S' (\rn)$ such that
\begin{\eq}   \label{1.12}
\| f \, | \Bs (\rn) \|_{\vp} = \Big( \sum^\infty_{j=0} 2^{jsq} \big\| (\vp_j \wh{f})^\vee \, | L_p (\rn) \big\|^q \Big)^{1/q} 
\end{\eq}
is finite $($with the usual modification if $q= \infty)$. 
\end{definition}

\begin{remark}   \label{R1.2}
These are the well--known Besov spaces. They are quasi--Banach spaces which are independent of the above resolution of unity $\vp$
according to \eqref{1.8}--\eqref{1.10} (equivalent quasi--norms). The theory of these spaces and also of their counterparts $\Fs (\rn)$
with \eqref{1.11} (including $F^s_{\infty,q} (\rn)$) may be found in \cite{T20}. For further details one may consult \cite{T83, T92, T06} and the (historical) references within. In what follows we rely exclusively on the special Besov spaces
\begin{\eq}   \label{1.13}
B^s_p (\rn) = B^s_{p,p} (\rn) \qquad \text{where} \quad 1<p<\infty, \quad s\in \real,
\end{\eq}
and the (fractional) Sobolev spaces
\begin{\eq}   \label{1.14}
H^s (\rn) = H^s_2 (\rn) = B^s_2 (\rn), \qquad s\in \real.
\end{\eq}
We will use that $H^s (\rn)$ are Hilbert spaces whereas $B^s_p (\rn)$ are reflexive Banach spaces which are isomorphic to $\ell_p$
(related references may be found below when these properties are needed).
\end{remark}

\subsection[Pseudodiff, op.]{Pseudodifferential operators}    \label{S1.2}
Our second main ingredient are pseudodifferential operators. We collect some basic definitions and related assertions. We rely on some
mapping properties of the pseudodifferential operator $T_\tau$,
\begin{\eq}   \label{1.15}
(T_\tau f)(x) = \int_{\rn} e^{ix \xi} \tau (x,\xi) \wh{f} (\xi) \, \di \xi, \qquad x\in \rn,
\end{\eq}
of the H\"{o}rmander class $\Psi^\sigma_{1,\delta} (\rn)$ with $n\in \nat$, $\sigma \in \real$ and $0 \le \delta
\le 1$ where the $C^\infty$ symbol $\tau (x,\xi)$ satisfies for some constants $c_{\alpha, \gamma} \ge 0$,
\begin{\eq}   \label{1.16}
\big| D^\alpha_x D^\gamma_\xi \tau (x, \xi) \big| \le c_{\alpha, \gamma} (1+ |\xi|)^{\sigma - |\gamma| + \delta |\alpha|}, \qquad
x\in \rn, \quad \xi \in \rn,
\end{\eq}
$\alpha \in \nat^n_0$, $\gamma \in \nat^n_0$. These classes of
operators attracted a lot of attention since the 1970s (and even earlier). The standard references are \cite{Tay81}, \cite{Hor85}
and the related parts of \cite{Ste93}. There one finds basic assertions and many applications. 
Mapping properties of the operators $T_\tau \in \Psi^\sigma_{1,\delta} (\rn)$ with $0 \le \delta <1$ (excluding the exotic class)
in the spaces
$\As (\rn)$, $A \in \{B,F\}$, according to Definition \ref{D1.1} and its counterparts $\Fs (\rn)$ with $p<\infty$
have been studied in the 1980s and 1990s based on the technicalities
available at that time.  We refer the reader to\cite[Theorem 6.2.2, pp.\,258--261]{T92} based on \cite{Pai83, Tri87, Tor91}. It has 
been shown in \cite{Run85} that one can incorporate the exotic class $\delta =1$ without further assumptions if the smoothness $s$ in
the target space $\As (\rn)$, $A\in \{B,F \}$, is sufficiently large. This applies in particular to the spaces $\Bs (\rn)$ as introduced
in Definition \ref{D1.1} if $1\le p \le \infty$ and $s>0$. We returned recently in \cite{Tri22} to this topic based on wavelets in
full generality. There one finds further references and additional explanations. This will not be repeated here. We fix a special case which covers what will needed below. As usual nowadays, $T: A \hra B$ means that $T$ is a linear and continuous mapping from the
quasi--Banach space $A$ into the quasi--Banach space $B$. 

\begin{theorem}   \label{T1.3}
Let $1 \le p,q \le \infty$ and $s>0$. Let $T_\tau \in \Psi^0_{1, \delta} (\rn)$ with $0 \le \delta \le 1.$Then
\begin{\eq}   \label{1.17}
T_\tau: \quad \Bs (\rn) \hra \Bs (\rn).
\end{\eq}
\end{theorem}

\begin{remark}    \label{R1.4}
This assertion is covered by the above literature. It remains valid for any $s\in \real$ if one excludes the exotic case $\delta =1$. 
The step from the zero class $\Psi^0_{1,\delta} (\rn)$ to $\Psi^\sigma_{1,\delta} (\rn)$ with $\sigma \in \real$ will be done below by 
lifting.
\end{remark}

\section[traces, fractal diff. op.]{Compact traces and fractal differential operators}   \label{S2}
\subsection[traces]{Compact traces}    \label{S2.1}
A further ingredient for the spectral assertions of fractal pseudodifferential  operators below are compact mappings between 
function spaces measured in terms of approximation numbers  and entropy numbers. First we recall the related definitions on an abstract
level. All Banach spaces and quasi--Banach spaces in this paper are complex.

\begin{definition}   \label{D2.1}
Let $T: A \hra B$ be a linear and continuous mapping from the Banach space $A$ into the Banach space $B$. 
\cm
{\em (i)} Then the approximation number $a_k  (T)$ of $T$,
\begin{\eq}   \label{2.1}
a_k (T) = \inf \{ \|T-L \|: \ \rank L <k \}, \qquad k\in \nat,
\end{\eq}
is the infimum over all linear mappings $L: A \hra B$ of $\rank L <k$, where $\rank L$ is the dimension of the range of $L$.
\cm
{\em (ii)} Then the entropy number $e_k (T)$, $k\in \nat$, is the infimum of all $\ve >0$ such that
\begin{\eq}   \label{2.2}
T (U_A) \subset \bigcup^{2^{k-1}}_{j=1} (b_j + \ve U_B ) \quad \text{for some} \quad b_1, \ldots b_{2^{k-1}} \in B,
\end{\eq}
where $U_A = \{ a\in A: \|a \,| A \| \le 1 \}$ and $U_B = \{ b\in B: \,\|b\,| B \| \le 1 \}$.
\end{definition}

\begin{remark}   \label{R2.2}
Details and references  may be found in \cite[Definition 1.87, Remark 1.88, Section 1.10, pp.\,55--58]{T06} and \cite[Section 1.3.1,
pp.~7--13]{ET96}. We assume that the reader is familiar  with these notation. Some related personal comments are given in Section \ref{S4} at the end of this paper. We collect very few properties which will be needed below. Obviously, $a_1 (T) = \|T \|.$ If $T: A \hra B$
is compact then the dual operator $T'$ is also compact and
\begin{\eq}    \label{2.3}
a_k (T') = a_k (T), \qquad T': B' \hra A', \quad k\in \nat.
\end{\eq}
A proof may  be found in \cite[Proposition 2.5, p.~55]{EdE87}. Duality assertions for entropy numbers are not known in general (as far as we know). But a special case plays a crucial role below. Then specific references will be given.
Furthermore the linear and continuous mapping $T: A \hra B$ is compact if, and only if, $e_k (T) \to 0$ for $k \to \infty$. Let $R,S,T$
be linear and continuous mappings. Then one has
\begin{\eq}   \label{2.4}
a_k (R \circ S \circ T) \le \|R \|\cdot a_k (S) \cdot \|T \|, \qquad k \in \nat,
\end{\eq}
\begin{\eq}   \label{2.4a}
a_{k+l-1} (S \circ T) \le a_k (S) \cdot a_l (T), \qquad k\in \nat, \quad l \in \nat.
\end{\eq}
and
\begin{\eq}   \label{2.5}
e_k (R \circ S \circ T) \le \|R \| \cdot e_k (S)  \cdot \|T \|, \qquad k \in \nat,
\end{\eq}
\begin{\eq}   \label{2.5a}
e_{k+l-1} (S \circ T) \le e_k (S) \cdot e_l (T), \qquad k\in \nat, \quad l \in \nat.
\end{\eq}
for related compositions.
\end{remark}

We assume that the reader is also familiar with the basic assertions of fractal geometry, especially fractal sets in \rn, inclusively so called $d$--sets, $0<d<n$, and related Hausdorff measures. All what one needs may be found in \cite[Sections 1.12 and 1.17]{T06}.
First we repeat some assertions obtained in \cite{T06} adapted to our further needs. This means in particular that we do not deal with
most general compactly supported isotropic Radon measures in $\rn$ as there, but focus on the distinguished sub--class of compact
$d$--sets. 

Let $\Gamma$ be a compact $d$--set in \rn, $0<d<n$, furnished with the Hausdorff measure $\mu = {\mathcal H}^d|\Gamma$. Let $L_p (
\Gamma, \mu)$ with $1<p<\infty$ be the usual complex Banach space normed by
\begin{\eq}   \label{2.6}
\|f\, | L_p (\Gamma,\mu)\| = \Big( \int_{\rn} |f(x)|^p \, \mu(\di x) \Big)^{1/p} = \Big( \int_{\Gamma} |f(\gamma)|^p \, \mu( \di \gamma
)\Big)^{1/p}.
\end{\eq}
As detailed in \cite[Sections 1.12.2, 1.17.2]{T06} one can interpret $L_p (\Gamma, \mu)$ as a subset of $S'(\rn)$ identifying $f\in
L_p (\Gamma, \mu)$ with the complex finite measure $f \mu \in S'(\rn)$,
\begin{\eq}   \label{2.7}
\big( \id_\mu f)(\vp) = \int_{\Gamma} f(\gamma) \, \vp (\gamma) \, \mu (\di \gamma), \qquad \vp \in S(\rn).
\end{\eq}
In addition to properties of these so--called {\em identification operators} $\id_\mu$ we are interested in traces on $\Gamma$. Let
$s>0$, $1<p<\infty$, and
\begin{\eq}   \label{2.8}
\Big( \int_{\Gamma} |\vp(\gamma)|^p \, \mu (\di \gamma) \Big)^{1/p} \le c \, \| \vp \, | B^s_p (\rn) \|
\end{\eq}
for some $c>0$ and all $\vp \in S(\rn)$. Here $B^s_p (\rn)$ are the special Besov spaces according to \eqref{1.13}.
Then the completion of the pointwise trace $\big(\tr_{\mu} \vp \big)(\gamma) =\vp (\gamma)$, 
$\gamma \in \Gamma$, in $B^s_p (\rn)$ is called the trace and
\begin{\eq}   \label{2.9}
\tr_\mu: \quad B^s_p (\rn) \hra L_p (\Gamma, \mu)
\end{\eq}
is the corresponding linear and bounded {\em trace operator}. A detailed discussion in the context of more general measures and spaces
may be found in \cite[Section 1.17.2 and Chapter 7]{T06}. We collect some properties of the operators $\tr_\mu$ and $\id_\mu$ specifying 
corresponding assertions in \cite{T06} in the context of related approximation numbers and entropy numbers as introduced in Definition
\ref{D2.1}. Let $\{b_i: i\in I \}$ and $\{ d_i: i\in I \}$, be two sets of positive numbers. Then $b_i \sim d_i$ (equivalence) means that there are two positive numbers $c_1$ and $c_2$ such that $c_1 b_i \le d_i \le c_2 b_i$ for all $i\in I$
where $I$ is an arbitrary index set.

\begin{theorem}  \label{T2.3}
Let $\Gamma$ be a compact $d$--set in \rn, $0<d<n$, and let $\mu$ be the corresponding Hausdorff measure.  Let
\begin{\eq}   \label{2.10}
1<p<\infty, \quad \text{and} \quad \frac{n-d}{p} <s \le \frac{n}{p}.
\end{\eq}
{\em (i)} Then $\tr_\mu$,
\begin{\eq}   \label{2.11}
\tr_\mu: \quad B^s_p (\rn) \hra L_p (\Gamma, \mu)
\end{\eq}
is compact and
 \begin{\eq}   \label{2.12}
e_k (\tr_\mu ) \sim a_k (\tr_\mu) \sim k^{- \frac{1}{p} + \frac{1}{d} (\frac{n}{p} -s )}, \qquad k\in \nat.
\end{\eq}
{\em (ii)} Then $\id_\mu$ is dual to $\tr_\mu$,
\begin{\eq}    \label{2.13}
\id_\mu = \tr_\mu': \quad L_{p'} (\Gamma, \mu) \hra B^{-s}_{p'} (\rn), \qquad \frac{1}{p} + \frac{1}{p'} = 1.
\end{\eq}
\end{theorem}

\begin{proof}
Part (i) is covered by \cite[Example 7.25, p.\,313]{T06} (as the specification of a more general assertion). The duality \eqref{2.13}
is a special case of \cite[(7.37), p.\,304]{T06} based on
\begin{\eq}   \label{2.14}
L_p (\Gamma, \mu)' = L_{p'} (\Gamma, \mu) \quad \text{and} \quad B^s_p (\rn)' = B^{-s}_{p'} (\rn)
\end{\eq}
where the latter is covered by \cite[Theorem 2.11.2, p.\,178]{T83}. One may also consult \cite[Section 9.2, pp.\,122--125]{T01}.
\end{proof}

\subsection[Fractal diff. op.]{Fractal differential operators}   \label{S2.2}
Let  
\begin{\eq}   \label{2.15}
w_\alpha (x) = (1 + |x|^2)^{\alpha/2}, \qquad \alpha \in \real, \quad x\in \rn,
\end{\eq}
and
\begin{\eq}   \label{2.16}
I_\alpha: \quad f \mapsto (w_\alpha \wh{f} )^\vee = (w_\alpha f^\vee )^\wedge, \qquad f\in S'(\rn), \quad \alpha \in \real.
\end{\eq}
Quite obviously, $I_\alpha$ maps $S(\rn)$ one--to--one onto itself and also $S'(\rn)$ one--to--one onto itself.
Furthermore, $I_\alpha$,
\begin{\eq}   \label{2.17}
I_\alpha \As (\rn) = A^{s-\alpha}_{p,q} (\rn), \quad
\| (w_\alpha \wh{f})^\vee | A^{s-\alpha}_{p,q} (\rn) \| \sim \|f \, | \As (\rn) \|,
\end{\eq}
is a lift in all spaces $\As (\rn)$ with $A \in \{B,F \}$, $0<p \le \infty$, $0<q \le \infty$, $s\in \real$ and $\alpha \in \real$.
A proof (in this generality, incorporating $F^s_{\infty,q} (\rn)$) may be found in \cite[Section 1.3.2, pp.\,16--17]{T20}. But
otherwise it is a very classical assertion. Here we combine these observations with the mapping properties of the operators $\tr_\mu$
and $\id_\mu$ according to Theorem \ref{T2.3} in the framework of the spaces $B^s_p (\rn)$ with $1<p<\infty$ and $s\in \real$. First
we collect what is already known in the case of the Hilbert spaces $H^s (\rn) =B^s_2 (\rn)$ as introduced in \eqref{1.14}. Let $\Gamma$
be again a compact $d$--set in \rn, $0<d<n$ and let $n-d <2s \le n$. Then one has by Theorem \ref{T2.3} that
\begin{\eq}   \label{2.18}
\id^\mu = \id_\mu \circ \tr_\mu: \quad H^s (\rn) \hra H^{-s} (\rn),
\end{\eq}
is a compact operator. Combined with \eqref{2.17}, specified to $\As (\rn) = H^s (\rn)$, it follows that
\begin{\eq}   \label{2.19}
D^\mu_s = I_{-2s} \circ \id^\mu = (\id - \Delta)^{-s} \circ \id^\mu: \quad H^s (\rn) \hra H^s (\rn)
\end{\eq}
is a compact operator. We dealt in \cite[Theorem 7.68, p.\,343]{T06} with these {\em fractal PDE}s in the context of these Hilbert spaces
$H^s (\rn) = B^s_2 (\rn)$. Specifying the underlying isotropic Radon measures to the above Hausdorff  measures one obtains the 
following assertion.

\begin{theorem}   \label{T2.4}
Let $\Gamma$ be a compact $d$--set in \rn, $0<d<n$, and let
\begin{\eq}    \label{2.20}
n-d <2s \le n.
\end{\eq}
Then $D^\mu_s$  according to \eqref{2.19} is a compact, non--negative self--adjoint  operator in $H^s (\rn)$. Let $\lambda_k$ be the positive eigenvalues of $D^\mu_s$,
repeated according to multiplicity and ordered by decreasing magnitude,
\begin{\eq}   \label{2.21}
\lambda_1 \ge \lambda_2 \ge \cdots >0, \qquad \lambda_k \to 0 \quad \text{if} \quad k \to \infty.
\end{\eq}
Then
\begin{\eq}   \label{2.22}
\lambda_k \sim k^{-1 + \frac{1}{d} (n-2s)}, \qquad k\in \nat.
\end{\eq}
\end{theorem}

\begin{remark}   \label{R2.5}
This is a special case of \cite[Theorem 7.68, p.\,343]{T06} with $H(t) = t^{1/d}$ in the notation used there, \cite[p.\,313]{T06}.
The proof is based on Hilbert space arguments, including
\begin{\eq}   \label{2.23}
|\lambda_k (T)| = a_k(T), \qquad k \in \nat,
\end{\eq} 
for the real eigenvalues $\lambda_k (T)$ and the approximation numbers $a_k (T)$ for compact self--adjoint operators in a Hilbert
space, \cite[Theorem 6.21, p.\,195]{HT08} or \cite[p.\,58]{T06} and the references given there (it is a very classical cornerstone
of the spectral theory in Hilbert spaces).
\end{remark}

\section[Fractal pseudodiff. op.]{Fractal pseudodifferential operators}  \label{S3}
\subsection[Preparations]{Preparations}    \label{S3.1}
It is the main aim of this paper to extend the spectral theory for the fractal differential operators $D^\mu_s$ according to 
\eqref{2.19} in the Hilbert spaces $H^s (\rn)$ to fractal pseudodifferential operators in some Banach spaces $B^s_p (\rn)$ as introduced in
\eqref{1.13}. In addition to Theorem \ref{T1.3} we rely essentially on two new ingredients. First we need a refinement of Theorem
\ref{T2.3}. Secondly we ask for a substitute  of \eqref{2.23}.

We complement  \eqref{2.12} by a corresponding assertion for the identification operator $\id_\mu$ using the duality \eqref{2.13}. For
the approximation numbers $a_k (T)$ and $a_k (T')$ one can rely on \eqref{2.3}. But there is no counterpart for corresponding entropy
numbers. According to  \cite[p.\,332]{Pie07} it is one of the most interesting challenges of operator
theory to compare the asymptotic behaviour of $e_k (T)$ and $e_k (T')$, where again $T:A \hra B$ is linear and compact mapping from
the Banach space $A$ into the Banach space $B$ and
$T': B' \hra A'$ is its dual. In particular one
may ask whether there are numbers $l \in \nat$ and $c \ge 1$ such that
\begin{\eq}   \label{3.1}
e_{lk} (T') \le c \, e_k (T), \qquad k\in \nat.
\end{\eq}
This problem is still unsolved in general (as far as we know). If both spaces $A$ and $B$ are Hilbert spaces then one has $e_k (T') =
e_k (T)$. A short proof and references may be found in \cite[Theorem 1.3.1, pp.\,9--10]{ET96}. Further affirmative assertions have
been obtained in \cite{AMS04} and \cite{AMST04}. In particular, according to \cite{AMST04} the inequality \eqref{3.1} has been
confirmed if at least one of the two Banach spaces $A$ and $B$ is isomorphic to $\ell_p$, $1<p<\infty$. We refer again to 
\cite[pp.\,332--333]{Pie07} where the somewhat different formulations in \cite{AMS04} and \cite{AMST04} have been reformulated in terms of
entropy numbers as in \eqref{3.1}. According to \cite[Corollary 3.8, p.~157]{T06}
the spaces $B^s_p (\rn) = B^s_{p,p} (\rn)$, $s\in \real$, $0<p \le \infty$ are isomorphic to $\ell_p$. This is an immediate consequence
of related wavelet isomorphisms. Now we can complement Theorem \ref{T2.3} as follows. Let $\tr_\mu$ and $\id_\mu$ be as there based on
\eqref{2.7} and \eqref{2.9}.

\begin{corollary}   \label{C3.1}
Let $\Gamma$ be a compact $d$--set in \rn, $0<d<n$, and let $\mu$ be the corresponding Hausdorff measure.
Let $1<p<\infty$, $\frac{1}{p} + \frac{1}{p'} =1$, 
\begin{\eq}   \label{3.2}
n-d <sp \le n \quad \text{and} \quad n-d < s' p' \le n.
\end{\eq}
Then both
\begin{\eq}   \label{3.3}
\tr_\mu: \quad B^s_p (\rn) \hra L_p (\Gamma, \mu)
\end{\eq}
and
\begin{\eq}   \label{3.4}
\id_\mu: \quad L_p (\Gamma, \mu) \hra B^{-s'}_p (\rn)
\end{\eq}
are compact. Furthermore
\begin{\eq}   \label{3.5}
e_k (\tr_\mu) \sim  a_k (\tr_\mu) \sim k^{- \frac{1}{p} + \frac{1}{d} (\frac{n}{p} - s )}, \qquad k\in \nat,
\end{\eq} 
and
\begin{\eq}   \label{3.6}
e_k (\id_\mu) \sim a_k (\id_\mu) \sim  k^{- \frac{1}{p'} + \frac{1}{d} (\frac{n}{p'} - s')}, \qquad k \in \nat.
\end{\eq}
\end{corollary}

\begin{proof} 
We deal with the entropy numbers $e_k$. The proof for the approximation numbers $a_k$ is the same, based on the duality \eqref{2.3}
instead of \eqref{3.1}. 
In addition to \eqref{2.12} = \eqref{3.5} it follows from Theorem \ref{T2.3} that
\begin{\eq}   \label{3.7}
e_k \big( \tr_\mu: \ B^{s'}_{p'} (\rn) \hra L_{p'} (\Gamma, \mu) \big) \sim k^{-\frac{1}{p'} + \frac{1}{d} (\frac{n}{p'} - s')},
\qquad k \in \nat.
\end{\eq}
The above discussion shows that we can rely on the duality assertion \eqref{3.1}.
Then one has by \eqref{2.13} (where $p$ and $p'$ change their roles) that
\begin{\eq}  \label{3.8}
e_k (\id_\mu) \le c \, k^{- \frac{1}{p'} + \frac{1}{d} (\frac{n}{p'} - s')}, \qquad k \in \nat,
\end{\eq}
for $\id_\mu$ as in \eqref{3.4}. All spaces in \eqref{3.4} and \eqref{3.7} are reflexive Banach spaces. Then \eqref{2.13} can
be complemented by
\begin{\eq}   \label{3.9}
\tr_\mu = \id_\mu': \quad B^{s'}_{p'} (\rn) \hra L_{p'} (\Gamma, \mu).
\end{\eq}
Now it follows from the duality assertion \eqref{3.1} and the equivalence \eqref{3.7} that \eqref{3.8} must be an 
equivalence. This proves \eqref{3.6} for the entropy numbers.
\end{proof}

Next we describe a substitute for \eqref{2.23}. Let $B$ be a complex infinitely--dimensional quasi--Banach space and let $K: \ B\hra B$ be a linear compact operator. Then its spectrum
in the complex plane $\comp$
consists of the origin $0$ and an at most countably infinite number of non--zero eigenvalues of finite algebraic
multiplicity which may
accumulate only at the origin. Recall that the algebraic multiplicity of an eigenvalue $\lambda \not= 0$ of $K$ is the dimension of
\begin{\eq}   \label{3.10}
\big\{ b\in B: \ (K- \lambda \, \id )^k b=0 \ \text{for some $k\in \nat$} \big\}.
\end{\eq}
This well--known assertion is also covered by \cite[Theorem, p.\,5]{ET96}. Let $\{ \lambda_k (K) \}$ be the sequence of all non--zero
eigenvalues of $K$, repeated according to their algebraic multiplicity and ordered so that
\begin{\eq}   \label{3.11} 
|\lambda_1 (K)| \ge |\lambda_2 (K)| \ge \ldots \ge 0.
\end{\eq}
If $K$ has only $m <\infty$ non--zero
distinct eigenvalues and if $M$ is the sum of their algebraic multiplicities then we put $\lambda_j (K) =0$ for
$j>M$. Let $e_k (K)$ be the entropy numbers of $K$ introduced in Definition \ref{D2.1}(ii).  Then
\begin{\eq}   \label{3.12}
| \lambda_k (K)| \le \sqrt{2} \, e_k (K), \qquad k\in \nat,
\end{\eq}
is Carl's inequality, going back to \cite{Carl81} and \cite{CaT80}. Details and further explanations may be found in \cite[Section 
1.3.4, pp.\,18--22]{ET96} and \cite[Theorem 6.25, p.\,197]{HT08}. At the end of this paper we add in Section \ref{S4} a few related
personal comments.

\subsection[Main theorem]{Main theorem}   \label{S3.2}
First we formalize the fractal pseudodifferential operator $T^\mu_\tau$ already mentioned in the Abstract. Afterwards we combine
Corollary \ref{C3.1} and suitable counterparts of \eqref{2.18}, \eqref{2.19}.

Let $F$ be the Fourier transform according to \eqref{1.6}. Let $L_p (\Gamma,\mu)$, $1<p<\infty$, be as in \eqref{2.6}, \eqref{2.7},
interpreting $f\mu$ for $f\in L_p (\Gamma, \mu)$ as a complex measure in \rn. Then one has for $\vp \in S(\rn)$ as usual
\begin{\eq}   \label{3.13}
\begin{aligned}
\big( F(f\mu), \vp\big) &= (f\mu, F\vp) \\
&= (2\pi)^{-\frac{n}{2}} \int_\Gamma f(\gamma) \int_{\rn} e^{-i \gamma x} \vp (x) \, \di x \, \mu (\di \gamma) \\
&= (2\pi)^{-\frac{n}{2}} \int_{\rn} \int_\Gamma  e^{-i \gamma x} f(\gamma)\, \mu (\di \gamma) \vp (x) \, \di x 
\end{aligned}
\end{\eq}
resulting in the analytic function
\begin{\eq}   \label{3.14}
F \big(f\mu\big)(x) = (2\pi)^{-\frac{n}{2}}  \int_\Gamma  e^{-i \gamma x} f(\gamma) \, \mu (\di \gamma), \qquad f\in L_p (\Gamma, \mu).
\end{\eq}
Let  $T_\tau$,
\begin{\eq}   \label{3.15}
\big( T_\tau f \big)(x) = \int_{\rn} e^{-i x \xi} \tau (x,\xi) f^\vee(\xi) \, \di \xi, \qquad x\in \rn,
\end{\eq}
be again the pseudodifferential operator of the H\"{o}rmander class $\Psi^\sigma_{1,\delta} (\rn)$, $\sigma \in \real$, $0 \le \delta \le 1$,
as already described in \eqref{1.15} (changing the role of $F$ and $F^{-1}$, but this is immaterial) where the $C^\infty$
symbol $\tau (x,\xi)$ satisfies for some constants
$c_{\alpha,\gamma} \ge 0$,
\begin{\eq}    \label{3.16}
\big| D^\alpha_x D^\gamma_{\xi} \tau (x,\xi) \big| \le c_{\alpha, \gamma} \big( 1 + |\xi| \big)^{\sigma - |\gamma| + \delta |\alpha|},
\quad x\in \rn, \quad \xi \in \rn.
\end{\eq}
We rely on the mapping property 
\begin{\eq}   \label{3.17}
T_\tau: \quad B^s_p (\rn) \hra B^s_p (\rn), \qquad s>0, \quad 1<p<\infty,
\end{\eq}
for the pseudodifferential operators $T_\tau \in \Psi^0_{1,\delta} (\rn)$, $0 \le \delta \le 1$, covered by Theorem \ref{T1.3}. 
The counterpart of \eqref{2.18} now based on Corollary \ref{C3.1} is given by
\begin{\eq}   \label{3.18}
\id^\mu = \id_\mu \circ \tr_\mu: \quad B^s_p (\rn) \hra B^{-s'}_p (\rn)
\end{\eq}
for suitably chosen $s$ and $s'$
This will be combined with $T_\tau$ in \eqref{3.15} such that the corresponding {\em fractal pseudodifferential
operator} $T^\mu_\tau,$
\begin{\eq}   \label{3.19}
\big( T^\mu_\tau f\big)(x) = \int_{\rn} e^{-ix\xi} \, \tau(x,\xi) \, \big( \id^\mu f \big)^\vee (\xi) \, \di \xi, \qquad x\in \rn,
\end{\eq}
is compact in $B^s_p (\rn)$. The counting of (possible) non--zero eigenvalues $\lambda_k (T^\mu_\tau)$ is the same as in \eqref{3.10},
\eqref{3.11}. We rely again on \eqref{3.12}.

\begin{theorem*}
Let $\Gamma$ be a compact $d$--set in \rn, $0<d<n$. Let $1<p<\infty$ and $n-d <sp \le n$. Then $T^\mu_\tau$ according to \eqref{3.19}
with $\sigma = -sp$ in \eqref{3.16} is a linear and compact operator in $B^s_p (\rn)$ and
\begin{\eq}   \label{3.20}
|\lambda_k (T^\mu_\tau)| \le c \, k^{-1 + \frac{1}{d} (n-sp)}
\end{\eq}
for some $c>0$ and all $k\in \nat$.
\end{theorem*}

\begin{proof}
We apply Corollary \ref{C3.1} with
\begin{\eq}   \label{3.21}
s' = \frac{sp}{p'} = s(p-1), \qquad \frac{n}{p'} - s' = \frac{1}{p'} (n-sp),
\end{\eq}
to \eqref{3.18}. Then
\begin{\eq}   \label{3.22}
\id^\mu: \quad B^s_p (\rn) \hra B^{-s(p-1)}_p (\rn)
\end{\eq}
is compact and
\begin{\eq}  \label{3.23}
e_k (\id^\mu ) \le c \, k^{-1 + \frac{1}{d} (n-sp)}, \qquad k \in \nat,
\end{\eq}
where we used \eqref{3.5}, \eqref{3.6} and the composition property \eqref{2.5a}. By \eqref{3.15} and \eqref{3.16} one has for $\sigma=
-sp$,
\begin{\eq}   \label{3.24}
T_{\tau_\sigma} \in \Psi^0_{1,\delta} (\rn), \quad \tau_\sigma (x,\xi)= \tau (x,\xi) \big( 1+ |\xi|^2 )^{\frac{sp}{2}}.
\end{\eq}
We apply the lift \eqref{2.16}, \eqref{2.17} to the right--hand side of \eqref{3.22},
\begin{\eq}    \label{3.25}
I_\sigma = I_{-sp} : \quad B^{-sp +s}_p (\rn) \hra B^s_p (\rn).
\end{\eq}
Then $T^\mu_\tau$ can be decomposed as
\begin{\eq}   \label{3.26}
T^\mu_\tau = T_{\tau_\sigma} \circ I_\sigma \circ \id^\mu: \quad B^s_p (\rn) \hra B^s_p (\rn)
\end{\eq}
based on \eqref{3.22}, \eqref{3.25} and Theorem \ref{T1.3}. Then \eqref{3.20} follows from \eqref{3.23}, the composition property
\eqref{2.5} and \eqref{3.12}.
\end{proof}

\begin{remark}   \label{R3.2}
The above Main Theorem extends for $p=2$ Theorem \ref{T2.4} to related fractal pseudodifferential operators. The non--negative 
selfadjoint operator $D^\mu_s$ in \eqref{2.19} coincides with \eqref{3.26} if $\tau_\sigma (x,\xi) =1$ and $p=2$ (then $\sigma = -2s$).
However in general one cannot expect equivalences as in \eqref{2.22}. But it shows that the exponent  $-1 + \frac{1}{d} (n-sp)$ in 
\eqref{3.20} is the best possible one. Spectral problems as considered in Theorem \ref{T2.4} and the Main Theorem attracted always some
attention. One may consult the recent papers \cite{RoS21, RoT22} and the references within. At the end of \cite{RoT22} the authors 
compare their results with related properties in \cite{ET96, T97} which in turn may be considered as forerunners of Theorem \ref{T2.4}
and Remark \ref{R2.5} as special cases of corresponding assertions in  \cite{T06} and the above Main Theorem.
\end{remark}

\section[Reminiscences]{Some personal reminiscences}   \label{S4}
The approximation numbers $a_k (T)$ according to Definition \ref{D2.1}(i) go back to the paper \cite{Pie63} of Pietsch published in a
lesser known journal. Nevertheless they had a great impact on the mathematical community round the world  dealing with (linear) 
operators in Banach spaces and quasi--Banach spaces including the problem  how to measure  their properties, especially compactness.
This was the major topic of Pietsch during his time at the Mathematical Department of the university in Jena. He had been appointed
there in 1965 and retired in 2000, but remained an active researcher in the years after. He studied in this time numerous already
known and newly created quantities, called $s$--{\em numbers}. Of special interest was the question to which so--called {\em operator 
ideals} the counterparts $s_k (T)$ of the above $a_k (T)$ belong. His respective research culminated in the two well--received books
\cite{Pie80} and \cite{Pie87}. The history of the entropy numbers is somewhat more involved. Instead of $e_k (T)$ in Definition 
\ref{D2.1}(ii) and inspired by a related suggestion in \cite{MP68} we introduced in \cite{Tri70} the entropy numbers $\ve_m (T)$,
$m\in \nat$, as the infimum of $\ve >0$ such that
\begin{\eq}   \label{4.1}
T (U_A) \subset \bigcup^{m}_{j=1} (b_j + \ve U_B ) \quad \text{for some} \quad b_1, \ldots b_m \in B,
\end{\eq}
(using the same notation as above). The motivation in \cite{Tri70} came from the proposal to introduce so--called {\em entropy ideals}
$E_{p,q}$ with $0<p<\infty$, $0<q \le \infty$,
\begin{\eq}  \label{4.2}
\begin{aligned}
E_{p,q} (A,B) &= \big\{ T:A\hra B, \Big(\sum^\infty_{m=2} \ve^q_m (T) (\log m)^{\frac{q}{p}-1} m^{-1} \Big)^{1/q} <\infty \big\}, \
q<\infty, \\
E_{p,\infty} (A,B) &= \big\{ T:A\hra B, \ \sup_{m \ge 2} \ve_m (T) (\log m)^{\frac{1}{p}} <\infty \big\}, \ q=\infty.
\end{aligned}
\end{\eq}
In a discussion with me Pietsch expressed some doubts. Even for $1\le p,q \le \infty$ these entropy ideals are only quasi--Banach
spaces (not necessarily Banach spaces). There is no duality, a cornerstone of Pietsch's theory of operator ideals. Later he gave up
his reservation. Secondly he suggested  to work  with the dyadic entropy numbers $\ve_{2^{k-1}} (T)$, $k \in \nat$, which coincide with
the entropy numbers $e_k (T)$ as introduced in Definition \ref{D2.1}(ii). Then \eqref{4.2} simplifies to
\begin{\eq}  \label{4.3}
\begin{aligned}
E_{p,q} (A,B) &= \big\{ T:A\hra B, \Big(\sum^\infty_{k=1} e^q_k (T) \,k^{\frac{q}{p}-1} \Big)^{1/q} <\infty \big\}, \
q<\infty, \\
E_{p,\infty} (A,B) &= \big\{ T:A\hra B, \ \sup_{k \ge 1} e_k (T) \, k^{\frac{1}{p}} <\infty \big\}, \ q=\infty.
\end{aligned}
\end{\eq}
Neither $\ve_m (T)$ nor $e_k (T)$ are $s$--numbers. But the version $e_k (T)$ brings them nearer to the axiomatically requested
properties of $s$--numbers and called later on {\em pseudo--$s$--numbers}. Further related discussions may be found in 
\cite[Section 12]{Pie87} and \cite[p.\,348]{Pie07}. His suggestion to give preference to the entropy numbers $e_k (T)$, compared with
$\ve_m (T)$, had been strikingly  underlined by the observation that these numbers can be used in the spectral theory of compact
operators. Let $K: B \hra B$ be a linear compact operator in the (complex) Banach space $B$ and let $\{\lambda_k (K) \}$ be the 
non--zero eigenvalues  counted with respect to their algebraic multiplicity as already described in  \eqref{3.10}, \eqref{3.11}. In
1979 B. Carl discovered that
\begin{\eq}     \label{4.4}
|\lambda_k (T)| \le \inf_{m \in \nat} 2^{\frac{m-1}{2k}} e_m (T), \qquad k \in \nat,
\end{\eq}
with \eqref{3.12} as a consequence. This has been published in \cite{Carl81}. Bernd Carl  is a pupil of Pietsch and was at this time 
a member of his research group. At the late 1980s he emigrated from East Germany (German Democratic Republic) to West Germany 
(Federal Republic of Germany). After the reunification of Germany he came back and was appointed 1992 in Jena where he stayed up to
his retirement in 2010. Discussing the above remarkable formula with me it came out more or less instantaneously (within hours) that
this assertion can be improved by
\begin{\eq}   \label{4.5}
\Big( \prod^k_{j=1} |\lambda_j (K) | \Big)^{1/k} \le \inf_{m\in \nat} 2^{\frac{m-1}{2k}} e_m (T), \qquad k\in \nat,
\end{\eq}
using geometric arguments, published in \cite{CaT80}. The proof in this short note became standard (with some modifications) and has been 
taken over in half a dozen books. The most elaborated version of this proof and further discussions of Carl's findings can be found
in \cite[Sections 6.3 -- 6.7]{HT08}. One may also consult \cite{CaS90} and \cite{ET96}.

The reputation of Pietsch in his early years originate from his approach to Grothendieck's theory of nuclear spaces and nuclear 
operators in \cite{Gro54, Gro55} and, especially, in \cite{Gro56}, being considered as ingenious but almost unreadable. Pietsch's
version in \cite{Pie65} (in German, but translated first in Russian and then in English) simplified this topic  significantly and
made it accessible to the common mathematician working in this field. In addition to his research he became in his later years more
and more interested in the history of Functional Analysis inclusively the persons behind their theorems. The related studies 
intensified after his retirement in 2000 and culminated finally in the encyclopaedic treatise \cite{Pie07} of 850 pages (including 
a Bibliography covering 50 pages). One would expect that he had helpers  collecting papers, producing drafts and other useful
preparations  for such an enormous  undertaking. But this is not the case. It is almost unbelievable, but it is the work of a single person: Albrecht Pietsch. What an achievement.

\end{document}